\documentclass[11pt, final]{article}
\usepackage{a4}
\usepackage{amsmath}%
\usepackage{amstext}%
\usepackage{amssymb}%
\usepackage{showkeys}%
\usepackage{epsfig}%
\usepackage{cite}
\usepackage{tikz}
\usepackage{subfig}
\usepackage{caption}
\usepackage{multirow}
\usepackage{booktabs}
\usepackage{floatrow}

\usepackage{listings}
\newtheorem{theorem}{Theorem}

\newtheorem{axiom}{Axiom}

\newtheorem{definition}[axiom]{Definition}

\newtheorem{lemma}[theorem]{Lemma}

\newtheorem{rem}[theorem]{Remark}
\newenvironment{remark}{\rem\rm}{\endrem}

\newcounter{unnumber}

\newtheorem{prf}[unnumber]{Proof.}
\newenvironment{proof}{\prf\rm}{\hfill{$\blacksquare$}\endprf}
\newcommand{\R}{\mathbb{R}}%
\newcommand{\N}{\mathbb{N}}%
\newcommand{\ol}{\overline}%

\renewcommand{\>}{\right\rangle}

\DeclareMathOperator*\dom{dom}%
\DeclareMathOperator*\B{\overline{\R}}%

\DeclareMathOperator*\crit{crit}
\DeclareMathOperator*\dist{dist}

\usepackage{geometry}
\title{Newton-like dynamics associated to nonconvex optimization problems}

\author{Radu Ioan Bo\c{t} \thanks{University of Vienna, Faculty of Mathematics, Oskar-Morgenstern-Platz 1, A-1090 Vienna, Austria,
email: radu.bot@univie.ac.at. Research partially supported by FWF (Austrian Science Fund), project I 2419-N32.} \and
Ern\"{o} Robert Csetnek \thanks {University of Vienna, Faculty of Mathematics, Oskar-Morgenstern-Platz 1, A-1090 Vienna, Austria,
email: ernoe.robert.csetnek@univie.ac.at. Research supported by FWF (Austrian Science Fund), project P 29809-N32.}}

\begin{document}
\maketitle

\noindent \textbf{Abstract.} We consider the dynamical system 
\begin{equation*}\left\{
\begin{array}{ll}
v(t)\in\partial\phi(x(t))\\
\lambda\dot x(t) + \dot v(t) + v(t) + \nabla \psi(x(t))=0,
\end{array}\right.\end{equation*}
where $\phi:\R^n\to\R\cup\{+\infty\}$ is a proper, convex and lower semicontinuous function, $\psi:\R^n\to\R$ is a (possibly nonconvex) smooth function 
and $\lambda>0$ is a parameter which controls the velocity. We show that the set of limit points of the trajectory $x$ is contained in the set 
of critical points of the objective function $\phi+\psi$, which is here seen as the set of the zeros of its limiting subdifferential. If the objective function satisfies the 
Kurdyka-\L{}ojasiewicz property, then we can prove convergence of the whole trajectory $x$ to a critical point. Furthermore, convergence rates for the orbits are obtained in terms of 
the \L{}ojasiewicz exponent of the objective function, provided the latter satisfies the \L{}ojasiewicz property.
\vspace{1ex}

\vspace{1ex}


\section{Introduction and preliminaries}\label{sec1}

The dynamical system 
\begin{equation}\label{att-sv-dyn}\left\{
\begin{array}{ll}
v(t)\in T(x(t))\\
\lambda(t)\dot x(t) + \dot v(t) + v(t)=0,
\end{array}\right.\end{equation}
where $\lambda:[0,+\infty)\to [0,+\infty)$ and $T:\R^n\rightrightarrows \R^n$ is a (set-valued) maximally monotone operator, 
has been introduced and investigated in \cite{att-sv2011} as a continuous version of Newton and Levenberg-Marquardt-type algorithms.
It has been shown that under mild conditions on $\lambda$ the trajectory $x(t)$ converges weakly to a zero of the operator $T$, while $v(t)$ converges to zero as $t \rightarrow +\infty$.

These investigations have been continued in \cite{abbas-att-sv} in the context of solving optimization problems of the form 
\begin{equation}\label{opt-intr}\inf_{x\in\R^n}\{\phi(x)+\psi(x)\},\end{equation}
where $\phi:\R^n\to\R\cup\{+\infty\}$ is a proper, convex and lower semicontinuous function and $\psi:\R^n\to\R$ is a convex  and differentiable function with locally Lipschitz-continuous gradient. More precisely, problem \eqref{opt-intr} has been approached via the dynamical system 
\begin{equation}\label{dyn-intr}\left\{
\begin{array}{ll}
v(t)\in\partial\phi(x(t))\\
\lambda(t)\dot x(t) + \dot v(t) + v(t) + \nabla \psi(x(t))=0,
\end{array}\right.\end{equation}
where $\partial\phi$ is the convex subdifferential of $\phi$. It has been shown in  \cite{abbas-att-sv} that if the set of minimizers of \eqref{opt-intr} is nonempty and 
some mild conditions on the damping function $\lambda$ are satisfied, then the trajectory $x(t)$ converges to a minimizer of \eqref{opt-intr} as $t \rightarrow +\infty$. Further investigations on dynamical systems of similar type have been reported in \cite{abbas} and \cite{b-c-lev-marq}. 

The aim of this paper is to perform an asymptotic analysis of the dynamical system \eqref{dyn-intr} in the absence of the convexity of $\psi$, for constant damping function $\lambda$ and by assuming that the objective function of \eqref{opt-intr} satisfies the \emph{Kurdyka-\L{}ojasiewicz} property, in other words is a  \emph{KL function}. To the class of KL functions belong semialgebraic, real subanalytic, uniformly convex and convex functions satisfying a growth condition. The convergence analysis relies on methods of real algebraic geometry introduced by \L{}ojasiewicz \cite{lojasiewicz1963} and Kurdyka \cite{kurdyka1998} and developed 
recently in the nonsmooth setting by Attouch, Bolte and Svaiter \cite{att-b-sv2013} and Bolte, Sabach and Teboulle \cite{b-sab-teb}. 

Optimization problems involving KL functions have attracted the interest of the community 
since the works of \L{}ojasiewicz \cite{lojasiewicz1963}, Simon \cite{simon}, Haraux and Jendoubi \cite{h-j}.  The most important contributions of the last years in the field include the works of
Alvarez, Attouch, Bolte and  Redont \cite[Section 4]{alv-att-bolte-red} and Bolte, Daniilidis and Lewis \cite[Section 4]{b-d-l2006}. Ever since the interest in this topic increased continuously 
(see \cite{attouch-bolte2009, att-b-red-soub2010, att-b-sv2013, b-n-p-s, b-sab-teb, b-c-kl, bcl, b-c-hidd, c-pesquet-r, f-g-peyp, 
h-l-s-t, ipiano}).

In the first part of the paper we show that the set of limit points of the trajectory $x$ generated by \eqref{dyn-intr} is entirely contained in the set of critical 
points of the objective function $\phi+\psi$, which is seen as the set of zeros of its limiting subdifferential. Under some 
supplementary conditions, including the Kurdyka-\L{}ojasiewicz property, we prove the convergence 
of the trajectory $x$ to a critical point of $\phi+\psi$. Furthermore, convergence rates for the orbits are obtained in terms of the \L{}ojasiewicz exponent of the objective function, provided the latter satisfies the \L{}ojasiewicz property.

In the following we recall  some notions and results which are needed throughout the paper. We consider on $\R^n$ the Euclidean scalar product 
and the corresponding norm denoted by $\langle\cdot,\cdot\rangle$ and $\|\cdot\|$, respectively. 

The {\it domain} of the function  $f:\R^n\rightarrow \R\cup\{+\infty\}$ is defined by $\dom f=\{x\in\R^n:f(x)<+\infty\}$ and  we say that $f$ is {\it proper}, if it has a nonempty domain. For the following generalized subdifferential notions and their basic properties we refer to \cite{borwein-zhu, boris-carte, rock-wets}. 
Let $f:\R^n\rightarrow \R\cup\{+\infty\}$ be a proper and lower semicontinuous function. The {\it Fr\'{e}chet (viscosity)  
subdifferential} of $f$ at $x\in\dom f$ is the set $$\hat{\partial}f(x)= \left \{v\in\R^n: \liminf_{y\rightarrow x}\frac{f(y)-f(x)-\<v,y-x\>}{\|y-x\|}\geq 0 \right \}.$$ If
$x\notin\dom f$, we set $\hat{\partial}f(x):=\emptyset$. The {\it limiting (Mordukhovich) subdifferential} is defined at $x\in \dom f$ by 
$$\partial_L f(x)=\{v\in\R^n:\exists x_k\rightarrow x,f(x_k)\rightarrow f(x)\mbox{ and }\exists v_k\in\hat{\partial}f(x_k),v_k\rightarrow v \mbox{ as }k\rightarrow+\infty\},$$
while for $x \notin \dom f$, we set $\partial_L f(x) :=\emptyset$. Obviously,  $\hat\partial f(x)\subseteq\partial_L f(x)$ for each $x\in\R^n$.

When $f$ is convex, these subdifferential notions coincide with the {\it convex subdifferential}, thus 
$\hat\partial f(x)=\partial_L f(x)=\partial f(x)=\{v\in\R^n:f(y)\geq f(x)+\<v,y-x\> \ \forall y\in \R^n\}$ for all $x\in\R^n$. 

The following {\it closedness criterion} of the graph of the limiting subdifferential will be used in the convergence analysis: 
if $(x_k)_{k\in\N}$ and $(v_k)_{k\in\N}$ are sequences in $\R^n$ such that 
$v_k\in\partial_L f(x_k)$ for all $k\in\N$, $(x_k,v_k)\rightarrow (x,v)$ and $f(x_k)\rightarrow f(x)$ as $k\rightarrow+\infty$, then 
$v\in\partial_L f(x)$. 

The Fermat rule reads in this nonsmooth setting as follows: if $x\in\R^n$ is a local minimizer of $f$, then $0\in\partial_L f(x)$.  We denote by 
$$\crit(f)=\{x\in\R^n: 0\in\partial_L f(x)\}$$ the set of {\it (limiting)-critical points} of $f$. 

When $f$ is continuously differentiable around $x \in \R^n$ we have $\partial_L f(x)=\{\nabla f(x)\}$. We will also make use of the following subdifferential sum rule:
if $f:\R^n\rightarrow\R\cup\{+\infty\}$ is proper and lower semicontinuous  and $h:\R^n\rightarrow \R$ is a 
continuously differentiable function, then $\partial_L (f+h)(x)=\partial_L f(x)+\nabla h(x)$ for all $x\in\R^n$. 

Further, we recall the notion of a locally absolutely continuous function and state two of its basic properties. 

\begin{definition}\label{abs-cont} \rm (see \cite{att-sv2011, abbas-att-sv}) 
A function $x : [0,+\infty) \rightarrow \R^n$ is said to be locally absolutely continuous, if it absolutely continuous on every interval $[0,T]$ for $T > 0$.
\end{definition}

\begin{remark}\label{rem-abs-cont}\rm\begin{enumerate} \item[(a)] An absolutely continuous function is differentiable almost 
everywhere, its derivative coincides with its distributional derivative almost everywhere and one can recover the function from its 
derivative $\dot x=y$ by integration. 

\item[(b)] If $x:[0,T]\rightarrow \R^n$ is absolutely continuous for $T > 0$ and $B:\R^n\rightarrow \R^n$ is 
$L$-Lipschitz continuous for $L\geq 0$, then the function $z=B\circ x$ is absolutely continuous, too.  Moreover, $z$ is differentiable almost everywhere on $[0,T]$ and the inequality 
$\|\dot z (t)\|\leq L\|\dot x(t)\|$ holds for almost every $t \in [0,T]$.  
\end{enumerate}
\end{remark}

The following two results, which can be interpreted as continuous versions of the quasi-Fej\'er monotonicity for sequences, 
will play an important role in the asymptotic analysis of the trajectories of the dynamical system \eqref{dyn-intr}. 
For their proofs we refer the reader  to \cite[Lemma 5.1]{abbas-att-sv} and \cite[Lemma 5.2]{abbas-att-sv}, respectively.

\begin{lemma}\label{fejer-cont1} Suppose that $F:[0,+\infty)\rightarrow\R$ is locally absolutely continuous and bounded from below and that
there exists $G\in L^1([0,+\infty))$ such that for almost every $t \in [0,+\infty)$ $$\frac{d}{dt}F(t)\leq G(t).$$ 
Then there exists $\lim_{t\rightarrow \infty} F(t)\in\R$. 
\end{lemma}

\begin{lemma}\label{fejer-cont2}  If $1 \leq p < \infty$, $1 \leq r \leq \infty$, $F:[0,+\infty)\rightarrow[0,+\infty)$ is 
locally absolutely continuous, $F\in L^p([0,+\infty))$, $G:[0,+\infty)\rightarrow\R$, $G\in  L^r([0,+\infty))$ and 
for almost every $t \in [0,+\infty)$ $$\frac{d}{dt}F(t)\leq G(t),$$ then $\lim_{t\rightarrow +\infty} F(t)=0$. 
\end{lemma}

The following result, which is due to Br\'{e}zis (\cite[Lemme 3.3, p. 73]{brezis}; see also \cite[Lemma 3.2]{att-cza-10}), 
provides an expression for the derivative of the composition of convex functions with absolutely continuous trajectories.

\begin{lemma}\label{diff-brezis} Let $f:\R^n\rightarrow \R\cup\{+\infty\}$ be a proper, convex and lower semicontinuous function. 
Let $x\in L^2([0,T],\R^n)$ be absolutely continuous such that $\dot x\in L^2([0,T],\R^n)$ and $x(t)\in\dom f$ for almost every 
$t \in [0,T]$. Assume that there exists $\xi\in L^2([0,T],\R^n)$ such that $\xi(t)\in\partial f(x(t))$ for almost every $t \in [0,T]$. Then the function 
$t\mapsto f(x(t))$ is absolutely continuous and for almost every $t$ such that $x(t)\in\dom \partial f$ we have 
$$\frac{d}{dt}f(x(t))=\langle \dot x(t),h\rangle \ \forall h\in\partial f(x(t)).$$ 
\end{lemma}

\section{Asymptotic analysis}\label{sec2}

In this paper we investigate the dynamical system 
\begin{equation}\label{dyn-syst}\left\{
\begin{array}{ll}
v(t)\in\partial \phi(x(t))\\
\lambda\dot x(t)+\dot v(t)+v(t)+\nabla \psi(x(t))=0\\
x(0)=x_0, v(0)=v_0\in\partial\phi(x_0),
\end{array}\right.\end{equation}
where $x_0,v_0\in\R^n$ and $\lambda>0$. 
We assume that $\phi:\R^n\rightarrow\R\cup\{+\infty\}$ is proper, convex and lower semicontinuous and
$\psi:\R^n\rightarrow\R$ is possibly nonconvex and Fr\'{e}chet differentiable with $L$-Lipschitz continuous gradient, 
for $L>0$; in other words, $\|\nabla\psi(x)-\nabla\psi(y)\|\leq L\|x-y\|$ for all $x,y\in \R^n$.

In the following we specify what we understand under a solution of the dynamical system \eqref{dyn-syst}.

\begin{definition}\label{str-sol}\rm Let $x_0,v_0\in\R^n$ and $\lambda>0$ be such that $v_0\in\partial\phi(x_0)$. We say that the pair $(x,v)$ is a strong global solution of \eqref{dyn-syst} if the following properties are satisfied: 
\begin{enumerate}
\item[(i)] $x,v:[0,+\infty)\rightarrow \R^n$ are locally absolutely continuous functions;

\item[(ii)] $v(t)\in\partial\phi(x(t))$ for every $t\in[0,+\infty)$;

\item[(iii)] $\lambda\dot x(t) + \dot v(t) + v(t) + \nabla \psi(x(t))=0$ for almost every $t\in[0,+\infty)$;

\item[(iv)] $x(0)=x_0, v(0)=v_0$.
\end{enumerate}
\end{definition}

The existence and uniqueness of the trajectories generated by \eqref{dyn-syst} has been investigated in 
\cite{abbas-att-sv}. A careful look at the proofs in \cite{abbas-att-sv} reveals the fact that the convexity of $\psi$ is not used in 
the mentioned results on the existence, but the Lipschitz-continuity of its gradient.

We start our convergence analysis with the following technical result.

\begin{lemma}\label{diff-Phix} Let $x_0,v_0\in\R^n$ and $\lambda>0$ be such that $v_0\in\partial\phi(x_0)$. 
Let $(x,v):[0,+\infty)\rightarrow\R^n \times \R^n$ 
be the unique strong global solution of the dynamical system \eqref{dyn-syst}. Then the following statements are true: 
\begin{enumerate}
 \item[(i)] $\langle \dot x(t),\dot v(t)\rangle\geq 0$ for almost every $t\in [0,+\infty)$;  
 \item[(ii)] $\frac{d}{dt}\phi(x(t))=\langle \dot x(t),v(t)\rangle$ for almost every $t\in [0,+\infty)$.
 \end{enumerate}
\end{lemma}

\begin{proof} (i) See \cite[Proposition 3.1]{att-sv2011}. The proof relies on the first relation in \eqref{dyn-syst} and 
the monotonicity of the convex subdifferential. 

(ii) The proof makes use of Lemma \ref{diff-brezis}. This relation has been already stated in \cite[relation (51)]{abbas-att-sv} without making use in its proof of the convexity of $\psi$. 
\end{proof}

\begin{lemma}\label{l-decr} Let $x_0,v_0\in\R^n$ and $\lambda>0$ be such that 
$v_0\in\partial\phi(x_0)$. Let $(x,v):[0,+\infty)\rightarrow\R^n \times \R^n$ 
be the unique strong global solution of the dynamical system \eqref{dyn-syst}. Suppose that $\phi+\psi$ is bounded from below. 
Then the following statements are true: 
\begin{enumerate}
\item [(i)] $\frac{d}{dt}(\phi+\psi)(x(t))+\lambda\|\dot x(t)\|^2+\langle \dot x(t),\dot v(t)\rangle=0$
for almost every $t\geq 0$;
\item [(ii)] $\dot x,\dot v, v+\nabla\psi(x)\in L^2([0,+\infty);\R^n)$, $\langle \dot x(\cdot),\dot v(\cdot)\rangle\in L^1([0,+\infty);\R)$ 
and $\lim_{t\rightarrow+\infty}\dot x(t)= \ \ $ $\lim_{t\rightarrow+\infty}\dot v(t)= \lim_{t\rightarrow+\infty}\big(v(t)+\nabla\psi(x(t))\big)=0$;
\item [(iii)] $\exists\lim_{t\rightarrow+\infty}(\phi+\psi)\big(x(t)\big)\in\R$.
\end{enumerate}
\end{lemma}

\begin{proof} (i) The statement follows by inner multiplying the both sides of the second relation in \eqref{dyn-syst} by $\dot x(t)$ and by taking afterwards into consideration Lemma \ref{diff-Phix}(ii). 

(ii) After integrating the relation (i) and by taking into account that $\phi+\psi$ is bounded from below, we easily derive 
$\dot x \in L^2([0,+\infty);\R^n)$ and $\langle \dot x(\cdot),\dot v(\cdot)\rangle\in L^1([0,+\infty);\R)$ 
(see also Lemma \ref{diff-Phix}(i)). 
Further, by using the second relation in \eqref{dyn-syst}, Remark \ref{rem-abs-cont}(b) and 
Lemma \ref{diff-Phix}(i), we obtain for almost every $t\geq 0$: 
\begin{eqnarray*}\frac{d}{dt}\left(\frac{1}{2}\left\|v(t)+\nabla\psi(x(t))\right\|^2\right) 
& = & \left\langle \dot v(t)+\frac{d}{dt}\nabla\psi(x(t)),v(t)+\nabla\psi(x(t)) \right\rangle\\
& = & \left\langle \dot v(t)+\frac{d}{dt}\nabla\psi(x(t)),-\lambda\dot x(t)-\dot v(t) \right\rangle\\
& = &-\lambda\langle \dot v(t),\dot x(t)\rangle-\|\dot v(t)\|^2-\lambda\left\langle \frac{d}{dt}\nabla\psi(x(t)), \dot x(t)\right\rangle\\
&   & -\lambda\left\langle \frac{d}{dt}\nabla\psi(x(t)), \dot v(t)\right\rangle\\
& \leq & -\|\dot v(t)\|^2-\lambda\left\langle \frac{d}{dt}\nabla\psi(x(t)), \dot x(t)\right\rangle
-\lambda\left\langle \frac{d}{dt}\nabla\psi(x(t)), \dot v(t)\right\rangle\\
& \leq & -\|\dot v(t)\|^2+\lambda L\|\dot x(t)\|^2+L\|\dot x(t)\|\cdot\|\dot v(t)\|\\
& \leq & -\|\dot v(t)\|^2+\lambda L\|\dot x(t)\|^2+L^2\|\dot x(t)\|^2+\frac{1}{4}\|\dot v(t)\|^2,
\end{eqnarray*}
hence 
\begin{equation}\label{ineq} \frac{d}{dt}\left(\frac{1}{2}\left\|v(t)+\nabla\psi(x(t))\right\|^2\right)+\frac{3}{4}\|\dot v(t)\|^2
\leq  L(\lambda+L)\|\dot x(t)\|^2.
\end{equation}
Since $\dot x\in L^2([0,+\infty);\R^n)$, a simple integration argument yields that $\dot v\in L^2([0,+\infty);\R^n)$. Considering 
the second equation in \eqref{dyn-syst}, we further obtain that $v+\nabla\psi(x)\in L^2([0,+\infty);\R^n)$. This fact combined with 
Lemma \ref{fejer-cont2} and \eqref{ineq} implies that $\lim_{t\rightarrow+\infty}\big(v(t)+\nabla\psi(x(t))\big)=0$. 
From the second equation in \eqref{dyn-syst} we obtain 
\begin{equation}\label{lim0} \lim_{t\rightarrow+\infty} \lambda \dot x(t)+\dot v(t) =0.
\end{equation}
Further, from Lemma \ref{diff-Phix}(i) we have for almost every $t\geq 0$
\begin{eqnarray*}\|\dot v(t)\|^2 &\leq & \lambda^2\|\dot x(t)\|^2+2\lambda\langle \dot x(t),\dot v(t)\rangle+\|\dot v(t)\|^2 = \|\lambda \dot x(t)+\dot v(t)\|^2,
\end{eqnarray*}
hence from \eqref{lim0} we get $\lim_{t\rightarrow+\infty}\dot v(t)=0$. Combining this with \eqref{lim0} we conclude 
that $\lim_{t\rightarrow+\infty}\dot x(t)=0$. 

(iii) From (i) and Lemma \ref{diff-Phix}(i) it follows that 
\begin{equation}\label{obj-decr}\frac{d}{dt}(\phi+\psi)(x(t))\leq 0\end{equation}
for almost every $t\geq 0$. The conclusion follows by applying Lemma \ref{fejer-cont1}. 
\end{proof}

\begin{lemma}\label{lim-crit} Let $x_0,v_0\in\R^n$ and $\lambda>0$ be such that 
$v_0\in\partial\phi(x_0)$. Let $(x,v):[0,+\infty)\rightarrow\R^n \times \R^n$ 
be the unique strong global solution of the dynamical system \eqref{dyn-syst}. Suppose that $\phi+\psi$ is bounded from below.  
Let $(t_k)_{k\in\N}$ be a sequence such that $t_k\rightarrow+\infty$ and $x(t_k)\rightarrow\ol x \in \R^n\mbox{ as }k\rightarrow+\infty$. 
Then $$0\in\partial_L(\phi+\psi)(\ol x).$$
\end{lemma}

\begin{proof} From the first relation in \eqref{dyn-syst} and the subdifferential sum rule of the limiting subdifferential we derive for any $k \in \N$
\begin{equation}\label{in-subdiff}v(t_k)+\nabla\psi(x(t_k))\in \partial \phi(x(t_k))+\nabla\psi(x(t_k))=\partial_L(\phi+\psi)(x(t_k)).
\end{equation}
Further, we have \begin{equation}\label{x-olx}x(t_k)\rightarrow\ol x \mbox{ as }k\rightarrow+\infty\end{equation} and 
(see Lemma \ref{l-decr}(ii)) \begin{equation}\label{limo-v}v(t_k)+\nabla\psi(x(t_k))\rightarrow 0 \mbox{ as }k\rightarrow+\infty.\end{equation}
According to the closedness property of the limiting subdifferential, the proof is complete as soon as we show that 
\begin{equation}\label{lim-func} (\phi+\psi)(x(t_k))\rightarrow (\phi+\psi)(\ol x) \mbox{ as }k\rightarrow+\infty.
\end{equation}
From \eqref{x-olx}, \eqref{limo-v} and the continuity of $\nabla\psi$ we get 
\begin{equation}\label{lim-vt} v(t_k)\rightarrow -\nabla\psi(\ol x) \mbox{ as }k\rightarrow+\infty. 
\end{equation}
Further, since $v({t_k})\in\partial\phi(x(t_k))$, we have 
$$\phi(\ol x)\geq \phi(x(t_k))+\langle v(t_k),\ol x-x(t_k)\rangle \ \forall k \in \N.$$
Combining this with \eqref{x-olx} and \eqref{lim-vt} we derive 
$$\limsup_{k\rightarrow+\infty}\phi(x(t_k))\leq \phi(\ol x). $$
A direct consequence of the lower semicontinuity of $\phi$ is the relation 
$$\lim_{k\rightarrow+\infty}\phi(x(t_k))= \phi(\ol x),$$
which combined with \eqref{x-olx} and the continuity of $\psi$ yields \eqref{lim-func}. 
\end{proof}

We define the {\it limit set of $x$} as
$$\omega (x):=\{\ol x\in\R^n:\exists t_k\rightarrow+\infty \mbox{ such that }x(t_k)\rightarrow\ol x \mbox{ as }k\rightarrow+\infty\}.$$
We use also the {\it distance function} to a set, defined for $A\subseteq\R^n$ as $\dist(x,A)=\inf_{y\in A}\|x-y\|$  
for all $x\in\R^n$.

\begin{lemma}\label{l} Let $x_0,v_0\in\R^n$ and $\lambda>0$ be such that 
$v_0\in\partial\phi(x_0)$. Let $(x,v):[0,+\infty)\rightarrow\R^n \times \R^n$ 
be the unique strong global solution of the dynamical system \eqref{dyn-syst}. Suppose that $\phi+\psi$ is bounded from below and 
$x$ is bounded. Then the following statements are true:
\begin{itemize}
\item[(i)] $\omega(x)\subseteq \crit(\phi+\psi)$; 
\item[(ii)] $\omega(x)$ is nonempty, compact and connected;
\item[(iii)] $\lim_{t\to+\infty}\dist\big(x(t),\omega(x)\big)=0$;
\item[(iv)] $\phi+\psi$ is finite and constant on $\omega(x)$. 
\end{itemize}
\end{lemma}

\begin{proof} Statement (i) is a direct consequence of Lemma \ref{lim-crit}. 

Statement (ii) is a classical result from \cite{haraux}. We also refer the reader  to the proof of Theorem 4.1 in 
\cite{alv-att-bolte-red}, where it is shown that the properties of $\omega(x)$ of being nonempty, compact and connected 
are generic for bounded trajectories fulfilling  $\lim_{t\rightarrow+\infty}{\dot x(t)}=0$. 

Statement (iii) follows immediately since $\omega(x)$ is nonempty. 

(iv) According to Lemma \eqref{l-decr}(iii), there exists $\lim_{t\rightarrow+\infty}(\phi+\psi)\big(x(t)\big)\in\R$. Let us 
denote by $l\in\R$ this limit. Take $\ol x\in\omega(x)$. Then there exists $t_k\rightarrow+\infty$ such that 
$x(t_k)\rightarrow\ol x$ as $k\rightarrow+\infty$. From the proof of Lemma \ref{lim-crit} we have that 
$(\phi+\psi)(x(t_k))\rightarrow (\phi+\psi)(\ol x) \mbox{ as }k\rightarrow+\infty$, hence $(\phi+\psi)(\ol x)=l$.
\end{proof}

\begin{remark}\label{cond-x-bound} 
Suppose that $\phi+\psi$ is coercive, in other words,
$$\lim_{\|u\|\rightarrow+\infty}(\phi+\psi)(u)=+\infty.$$ 
Let $x_0,v_0\in\R^n$ and $\lambda>0$ be such that 
$v_0\in\partial\phi(x_0)$. Let $(x,v):[0,+\infty)\rightarrow\R^n \times \R^n$ 
be the unique strong global solution of the dynamical system \eqref{dyn-syst}. Then $\phi+\psi$ is bounded from below and  
$x$ is bounded.  

Indeed, since $\phi+\psi$ is a proper, lower semicontinuous and coercive function, it follows that 
$\inf_{u\in\R^n}[\phi(u)+\psi(u)]$ is finite and the infimum is attained. Hence $\phi+\psi$ is bounded from below. 
On the other hand, from \eqref{obj-decr} it follows
$$(\phi+\psi)(x(T))\leq 
  (\phi+\psi)(x_0)  \ \forall T \geq 0.$$
Since $\phi+\psi$ is coercive, the lower level sets of $\phi+\psi$ are bounded, hence the above inequality yields 
that $x$ is bounded. Notice that in this case $v$ is bounded too, due to the relation 
$\lim_{t\rightarrow+\infty}\big(v(t)+\nabla \psi(x(t))\big)=0$ (Lemma \ref{l-decr}(ii)) and the Lipschitz continuity of $\nabla\psi$.
\end{remark}

\section{Convergence of the trajectory when the objective function satisfies the Kurdyka-\L{}ojasiewicz property}\label{sec3}

In order to enforce the convergence of the whole trajectory $x(t)$ to a critical point of the objective function as 
$t \rightarrow +\infty$ more involved analytic features of the functions have to be considered. 

A crucial role in the asymptotic analysis of the dynamical system \eqref{dyn-syst} is played by the class of functions satisfying the {\it Kurdyka-\L{}ojasiewicz} property. For $\eta\in(0,+\infty]$, we denote by $\Theta_{\eta}$ the class of concave and continuous functions 
$\varphi:[0,\eta)\rightarrow [0,+\infty)$ such that $\varphi(0)=0$, $\varphi$ is continuously differentiable on $(0,\eta)$, continuous at $0$ and $\varphi'(s)>0$ for all 
$s\in(0, \eta)$. 

\begin{definition}\label{KL-property} \rm({\it Kurdyka-\L{}ojasiewicz property}) Let $f:\R^n\rightarrow\R\cup\{+\infty\}$ be a proper and lower semicontinuous 
function. We say that $f$ satisfies the {\it Kurdyka-\L{}ojasiewicz (KL) property} at $\ol x\in \dom\partial_L f=\{x\in\R^n:\partial_L f(x)\neq\emptyset\}$, if there exist $\eta \in(0,+\infty]$, a neighborhood $U$ of $\ol x$ and a function $\varphi\in \Theta_{\eta}$ such that for all $x$ in the 
intersection 
$$U\cap \{x\in\R^n: f(\ol x)<f(x)<f(\ol x)+\eta\}$$ the following inequality holds 
$$\varphi'(f(x)-f(\ol x))\dist(0,\partial_L f(x))\geq 1.$$
If $f$ satisfies the KL property at each point in $\dom\partial_L f$, then $f$ is called {\it KL function}. 
\end{definition}

The origins of this notion go back to the pioneering work of \L{}ojasiewicz \cite{lojasiewicz1963}, where it is proved that for a real-analytic function 
$f:\R^n\rightarrow\R$ and a critical point $\ol x\in\R^n$ (that is $\nabla f(\ol x)=0$), there exists $\theta\in[1/2,1)$ such that the function 
$|f-f(\ol x)|^{\theta}\|\nabla f\|^{-1}$ is bounded around $\ol x$. This corresponds to the situation when $\varphi(s)=Cs^{1-\theta}$ for
$C>0$. The result of 
\L{}ojasiewicz allows the interpretation of the KL property as a re-parametrization of the function values in order to avoid flatness around the 
critical points. Kurdyka \cite{kurdyka1998} extended this property to differentiable functions definable in o-minimal structures. 
Further extensions to the nonsmooth setting can be found in \cite{b-d-l2006, att-b-red-soub2010, b-d-l-s2007, b-d-l-m2010}. 

One of the remarkable properties of the KL functions is their ubiquity in applications (see \cite{b-sab-teb}). We refer the reader to 
\cite{b-d-l2006, att-b-red-soub2010, b-d-l-m2010, b-sab-teb, b-d-l-s2007, att-b-sv2013, attouch-bolte2009} and the references therein for more properties of the KL functions and illustrating examples. 

In the analysis below the following uniform KL property given in \cite[Lemma 6]{b-sab-teb} will be used.

\begin{lemma}\label{unif-KL-property} Let $\Omega\subseteq \R^n$ be a compact set and let $f:\R^n\rightarrow\R\cup\{+\infty\}$ be a proper 
and lower semicontinuous function. Assume that $f$ is constant on $\Omega$ and that it satisfies the KL property at each point of $\Omega$.   
Then there exist $\varepsilon,\eta >0$ and $\varphi\in \Theta_{\eta}$ such that for all $\ol x\in\Omega$ and all $x$ in the intersection 
\begin{equation}\label{int} \{x\in\R^n: \dist(x,\Omega)<\varepsilon\}\cap \{x\in\R^n: f(\ol x)<f(x)<f(\ol x)+\eta\}\end{equation} 
the inequality  \begin{equation}\label{KL-ineq}\varphi'(f(x)-f(\ol x))\dist(0,\partial_L f(x))\geq 1.\end{equation}
holds.
\end{lemma}

Due to some reasons outlined in Remark \ref{techn} below, we prove the convergence of the trajectory $x(t)$ generated by \eqref{dyn-syst} as $t \rightarrow +\infty$ under the assumption that $\phi:\R^n\to \R$ is convex and differentiable with $\rho^{-1}$-Lipschitz continuous gradient for $\rho>0.$ In these circumstances the dynamical system \eqref{dyn-syst} reads
\begin{equation}\label{dyn-syst2}\left\{
\begin{array}{ll}
v(t)= \nabla\phi(x(t))\\
\lambda\dot x(t)+\dot v(t)+\nabla\phi(x(t))+\nabla \psi(x(t))=0\\
x(0)=x_0, v(0)=v_0=\nabla\phi(x_0),
\end{array}\right.\end{equation}
where $x_0,v_0\in\R^n$ and $\lambda>0$.

\begin{remark}\rm We notice that we do no require second order assumptions for $\phi$. However, we want to notice that if $\phi$ is a twice continuously differentiable function, then the  dynamical system  \eqref{dyn-syst2} can be equivalently written as 
\begin{equation}\label{dyn-syst2-c2}\left\{
\begin{array}{ll}
\lambda\dot x(t)+\nabla^2\phi(x(t))(\dot x(t))+\nabla\phi(x(t))+\nabla \psi(x(t))=0\\
x(0)=x_0, v(0)=v_0=\nabla\phi(x_0),
\end{array}\right.\end{equation}
where $x_0,v_0\in\R^n$ and $\lambda>0$. This is a differential equation with a Hessian-driven damping term. We refer the 
reader to \cite{alv-att-bolte-red} and \cite{att-mainge-red2012} for more insights into dynamical systems with Hessian-driven damping terms and for motivations for considering them.
Moreover, as in \cite{att-mainge-red2012}, the driving forces have been split as $\nabla\phi+\nabla\psi$, where $\nabla\psi$ stands for classical smooth driving forces and $\nabla\phi$ incorporates the contact forces. 
\end{remark}

In this context, an improved version of Lemma \ref{diff-Phix}(i) can be stated. 

\begin{lemma}\label{diff-Phix2} Let $x_0,v_0\in\R^n$ and $\lambda>0$ be such that $v_0=\nabla\phi(x_0)$. 
Let $(x,v):[0,+\infty)\rightarrow\R^n \times \R^n$ 
be the unique strong global solution of the dynamical system \eqref{dyn-syst2}. Then: 
\begin{equation}\label{impr}\langle \dot x(t),\dot v(t)\rangle\geq \rho\|\dot v(t)\|^2 \mbox{ for almost every }t\in [0,+\infty).\end{equation}
\end{lemma}

\begin{proof} Take an arbitrary $\delta>0$. For $t\geq 0$ we have 
\begin{eqnarray}\label{delta} \langle v(t+\delta)-v(t), x(t+\delta)-x(t)\rangle
& = & \langle \nabla\phi(x(t+\delta))-\nabla\phi (x(t)), x(t+\delta)-x(t)\rangle\nonumber\\
& \geq & \rho\|\nabla\phi(x(t+\delta))-\nabla\phi (x(t))\|^2\nonumber\\
& = & \rho\|v(t+\delta)-v(t)\|^2,
\end{eqnarray}
where the inequality follows from the Baillon-Haddad Theorem \cite[Corollary 18.16]{bauschke-book}. 
The conclusion follows by dividing \eqref{delta} by $\delta^2$ and by taking the limit as $\delta$ converges to zero from above. 
\end{proof}

We are now in the position to prove the convergence of the trajectories generated by \eqref{dyn-syst2}. 

\begin{theorem}\label{conv-kl} Let $x_0,v_0\in\R^n$ and $\lambda>0$ be such that 
$v_0=\nabla\phi(x_0)$. Let $(x,v):[0,+\infty)\rightarrow\R^n \times \R^n$ 
be the unique strong global solution of the dynamical system \eqref{dyn-syst2}. Suppose that $\phi+\psi$ is a KL function which 
is bounded from below and $x$ is bounded. Then the following statements are true:
\begin{itemize}\item[(i)] $\dot x,\dot v, \nabla\phi(x)+\nabla\psi(x)\in L^2([0,+\infty);\R^n)$, $\langle \dot x(\cdot),\dot v(\cdot)\rangle\in L^1([0,+\infty);\R)$ 
and \\$\lim_{t\rightarrow+\infty}\dot x(t)=\lim_{t\rightarrow+\infty}\dot v(t)= \lim_{t\rightarrow+\infty}\big(\nabla\phi(x(t))+\nabla\psi(x(t))\big)=0$;
\item[(ii)] there exists $\ol x\in\crit(\phi+\psi)$ (that is $\nabla(\phi+\psi)(\ol x)=0$) such that 
$\lim_{t\rightarrow+\infty}x(t)=\ol x$.
\end{itemize}
\end{theorem}

\begin{proof} According to Lemma \ref{l}, we can choose an element $\ol x\in\crit (\phi+\psi)$ (that is $\nabla(\phi+\psi)(\ol x)=0$) 
such that $\ol x\in \omega (x)$. According to Lemma \ref{l-decr}(iii), the proof of Lemma \ref{lim-crit} and the proof of Lemma \ref{l}(iv), 
we have 
$$\lim_{t\rightarrow+\infty}(\phi+\psi)(x(t))=(\phi+\psi)(\ol x).$$

We consider the following two cases.  

I. There exists $\ol t\geq 0$ such that $$(\phi+\psi)(x(\ol t))=(\phi+\psi)(\ol x).$$ 
From \eqref{obj-decr} we obtain for every $t\geq \ol t$ that
$$(\phi+\psi)(x(t))\leq(\phi+\psi)(x(\ol t))=(\phi+\psi)(\ol x)$$ 
Thus $(\phi+\psi)(x(t))=(\phi+\psi)(\ol x)$ for every $t\geq \ol t$. According to Lemma \ref{l-decr}(i) and \eqref{impr}, 
it follows that $\dot x(t)=\dot v(t)=0$ for almost every $t \in [\ol t, +\infty)$, hence $x$ and $v$ are constant on 
$[\ol t,+\infty)$ and the conclusion follows. 

II. For every $t\geq 0$ it holds $(\phi+\psi)(x(t))>(\phi+\psi)(\ol x)$. Take $\Omega:=\omega(x)$. 

By using Lemma \ref{l}(ii), (iv) and the fact that $\phi+\psi$ is a KL function, by Lemma \ref{unif-KL-property}, there exist positive numbers $\epsilon$ and $\eta$ and 
a concave function $\varphi\in\Theta_{\eta}$ such that for all $u$ belonging to the intersection
\begin{equation}\label{int-H}  \{u\in\R^n: \dist(u,\Omega)<\epsilon\}  
  \cap\left\{u\in\R^n:(\phi+\psi)(\ol x)<(\phi+\psi)(u)<(\phi+\psi)(\ol x)+\eta\right\},
 \end{equation}
one has
\begin{equation}\label{ineq-H}\varphi'\Big((\phi+\psi)(u)-(\phi+\psi)(\ol x)\Big)\cdot\|\nabla\phi (u)+\nabla\psi(u)\|\ge 1.\end{equation}

Let $t_1\geq 0$ be such that $(\phi+\psi)(x(t))<(\phi+\psi)(\ol x)+\eta$ for all $t\geq t_1$. Since 
$\lim_{t\to+\infty}\dist\big(x(t),\Omega\big)=0$ (see Lemma \ref{l}(iii)), there exists $t_2\geq 0$ such that for all $t\geq t_2$ the inequality 
$\dist\big(x(t),\Omega\big)<\epsilon$ holds. Hence for all $t\geq T:=\max\{t_1,t_2\}$, 
$x(t)$ belongs to the intersection in \eqref{int-H}. Thus, according to \eqref{ineq-H}, for every $t\geq T$ we have
\begin{equation}\label{ineq-Ht1}\varphi'\Big((\phi+\psi)(x(t))-(\phi+\psi)(\ol x)\Big)\cdot
\|\nabla\phi (x(t))+\nabla\psi(x(t))\|\ge 1.\end{equation}
From the second equation in \eqref{dyn-syst2} we obtain for almost every $t \in [T, +\infty)$
\begin{equation}\label{ineq-Ht2}(\lambda\|\dot x(t)\|+\|\dot v(t)\|)\cdot\varphi'\Big((\phi+\psi)(x(t))-(\phi+\psi)(\ol x)\Big)
\ge 1.\end{equation}

By using Lemma \ref{l-decr}(i), that $\varphi'>0$ and 
\begin{equation*}
\frac{d}{dt}\varphi\Big((\phi+\psi)(x(t))-(\phi+\psi)(\ol x)\Big)=
\varphi'\Big((\phi+\psi)(x(t))-(\phi+\psi)(\ol x)\Big)\frac{d}{dt}(\phi+\psi)(x(t)),
\end{equation*}
we further deduce that for almost every $t \in [T, +\infty)$ it holds
\begin{equation}\label{ineq-pt-conv-r} \frac{d}{dt}\varphi\Big((\phi+\psi)(x(t))-(\phi+\psi)(\ol x)\Big)\leq 
-\frac{\lambda\|\dot x(t)\|^2+\langle \dot x(t),\dot v(t)\rangle }{\lambda\|\dot x(t)\|+\|\dot v(t)\|}.\end{equation}
We invoke now Lemma \ref{impr} and obtain 
\begin{equation}\label{ineq-pt-conv-r-impr} \frac{d}{dt}\varphi\Big((\phi+\psi)(x(t))-(\phi+\psi)(\ol x)\Big)\leq 
-\frac{\lambda\|\dot x(t)\|^2+\rho\|\dot v(t)\|^2}{\lambda\|\dot x(t)\|+\|\dot v(t)\|}.\end{equation}

Let $\alpha>0$ (not depending on $t$) be such that 
\begin{equation}\label{ineq-alpha}-\frac{\lambda\|\dot x(t)\|^2+\rho\|\dot v(t)\|^2}{\lambda\|\dot x(t)\|+\|\dot v(t)\|}\leq
-\alpha\|\dot x(t)\|-\alpha\|\dot v(t)\| \ \forall t\geq 0.\end{equation}
One can for instance chose $\alpha>0$ such that 
$2\alpha\max(\lambda,1)\leq\min(\lambda,\rho)$. 
From \eqref{ineq-pt-conv-r-impr} we derive the inequality 
\begin{equation}\label{ineq-pt-conv-r2} \frac{d}{dt}\varphi\Big((\phi+\psi)(x(t))-(\phi+\psi)(\ol x)\Big)\leq 
-\alpha\|\dot x(t)\|-\alpha\|\dot v(t)\|,\end{equation}
which holds for almost every $t\geq T$. 
Since $\varphi$ is bounded from below, by integration it follows $\dot x,\dot v\in L^1([0,+\infty);\R^n)$. 
From here we obtain that $\lim_{t\rightarrow+\infty}x(t)$ exists and the conclusion follows from the results 
obtained in the previous section.
\end{proof}

\begin{remark}\label{techn}\rm Taking a closer look at the above proof, one can notice that the inequality 
\eqref{ineq-pt-conv-r} can be obtained also when $\phi:\R^n\to\R\cup\{+\infty\}$ is a (possibly nonsmooth) proper, convex and lower semicontinuous function. Though, in order to conclude that $\dot x\in L^1([0,+\infty);\R^n)$ the 
inequality obtained in Lemma \ref{diff-Phix}(i) is not enough. The improved version stated in Lemma \ref{diff-Phix2} is crucial 
in the convergence analysis. 

If one attempts to obtain in the nonsmooth setting the inequality stated in Lemma \ref{diff-Phix2}, from the proof of Lemma \ref{diff-Phix2} it becomes clear that one would need the inequality
$$\langle \xi_1^*-\xi_2^*,x_1-x_2\rangle\geq \rho\|\xi_1^*-\xi_2^*\|^2$$
for all $(x_1,x_2)\in\R^n\times\R^n$ and all $(\xi_1^*,\xi^*_2)\in\R^n\times\R^n$ such that $\xi_1^*\in\partial \phi(x_1)$ and 
$\xi_2^*\in\partial \phi(x_2)$. This is nothing else than (see for example \cite{bauschke-book})
$$\langle \xi_1^*-\xi_2^*,x_1-x_2\rangle\geq \rho\|\xi_1^*-\xi_2^*\|^2$$
for all $(x_1,x_2)\in\R^n\times\R^n$ and all $(\xi_1^*,\xi^*_2)\in\R^n\times\R^n$ such that $x_1\in\partial \phi^*(\xi_1^*)$ and 
$x_2\in\partial \phi^*(\xi_2^*)$. Here $\phi^*:\R^n\to\B$ denotes the Fenchel conjugate of $\phi$, defined for all $x^*\in\R^n$ by 
$\phi^*(x^*)=\sup_{x\in\R^n}\{\langle x^*,x\rangle-\phi(x)\}$. The latter inequality is equivalent to 
$\partial \phi^*$ is $\rho$-strongly monotone, which is further equivalent (see \cite[Theorem 3.5.10]{Zal-carte} or \cite{bauschke-book}) 
to $\phi^*$ is is strongly convex. This is the same with asking that $\phi$ is differentiable on the whole $\R^n$ with Lipschitz-continuous gradient (see 
\cite[Theorem 18.15]{bauschke-book}). In conclusion, the smooth setting provides the necessary prerequisites for obtaining the result in Lemma \ref{diff-Phix2} and, finally, Theorem \ref{conv-kl}.
\end{remark}

\section{Convergence rates}\label{sec4}

In this subsection we investigate the convergence rates of the trajectories $(x(t), v(t))$ generated by the dynamical system \eqref{dyn-syst2} as $t \rightarrow +\infty$. 
When solving optimization problems involving KL functions, convergence rates have been proved to depend on the so-called  \L{}ojasiewicz exponent 
(see \cite{lojasiewicz1963, b-d-l2006, attouch-bolte2009, f-g-peyp}). The main result of this subsection refers to the KL functions 
which satisfy Definition \ref{KL-property}  for $\varphi(s)=Cs^{1-\theta}$, where $C>0$ and $\theta\in(0,1)$. We recall the following 
definition considered in \cite{attouch-bolte2009}. 

\begin{definition}\label{kl-phi} \rm Let $f:\R^n\rightarrow\R\cup\{+\infty\}$ be a proper and lower semicontinuous function. 
The function $f$ is said to have the \L{}ojasiewicz property, if for every $\ol x\in\crit f$ there exist $C,\varepsilon >0$ and 
$\theta\in(0,1)$ such that 
\begin{equation}\label{kl-phi-ineq}|f(x)-f(\ol x)|^{\theta}\leq C\|x^*\| \ \mbox{for every} \ x \ \mbox{fulfilling} 
\ \|x-\ol x\|<\varepsilon \mbox{ and every} \ x^*\in\partial_L f(x).\end{equation}
\end{definition}

According to \cite[Lemma 2.1 and Remark 3.2(b)]{att-b-red-soub2010}, the KL property is automatically 
satisfied at any noncritical point, fact which motivates the restriction to critical points in the above definition. 
The real number $\theta$ in the above definition is called \emph{\L{}ojasiewicz exponent} of the function $f$ at the critical point  $\ol x$. 

The convergence rates obtained in the following theorem are in the spirit of \cite{b-d-l2006} and \cite{attouch-bolte2009}. 

\begin{theorem}\label{conv-r} Let $x_0,v_0\in\R^n$ and $\lambda>0$ be such that 
$v_0=\nabla\phi(x_0)$. Let $(x,v):[0,+\infty)\rightarrow\R^n \times \R^n$ 
be the unique strong global solution of the dynamical system \eqref{dyn-syst2}. 
Suppose that  $x$ is bounded and $\phi+\psi$ is a function which is bounded from below and satisfies Definition \ref{KL-property} 
for $\varphi(s)=Cs^{1-\theta}$, where $C>0$ and 
$\theta\in(0,1)$. Then there exists $\ol x\in\crit (\phi+\psi)$ (that is $\nabla(\phi+\psi)(\ol x)=0$) such that 
$\lim_{t\rightarrow+\infty}x(t)=\ol x$ and $\lim_{t\rightarrow+\infty}v(t)=\nabla\phi(\ol x)=-\nabla\psi(\ol x)$. Let $\theta$ be the 
\L{}ojasiewicz exponent of $\phi+\psi$ at $\ol x$, according to the Definition \ref{kl-phi}. Then there exist 
$a_1,b_1,a_2,b_2>0$ and $t_0\geq 0$ such that for every $t\geq t_0$ the following statements are true: 
\begin{itemize}\item[(i)] if $\theta\in (0,\frac{1}{2})$, then $x$ and $v$ converge in finite time;
\item[(ii)] if $\theta=\frac{1}{2}$, then $\|x(t)-\ol x\|+\|v(t)-\nabla\phi(\ol x)\|\leq a_1\exp(-b_1t)$;
\item[(iii)] if $\theta\in (\frac{1}{2},1)$, then $\|x(t)-\ol x\|+\|v(t)-\nabla\phi(\ol x)\|\leq (a_2t+b_2)^{-\left(\frac{1-\theta}{2\theta-1}\right)}$.
\end{itemize}
\end{theorem}

\begin{proof} According to the proof of Theorem \ref{conv-kl}, $\dot x,\dot v\in L^1([0,+\infty);\R^n)$ and there exists 
$\ol x\in\crit (\phi+\psi)$, in other words $\nabla(\phi+\psi)(\ol x)=0$, such that 
$\lim_{t\rightarrow+\infty}x(t)=\ol x$ and $\lim_{t\rightarrow+\infty}v(t)=\nabla\phi(\ol x)=-\nabla\psi(\ol x)$. 
Let $\theta$ be the \L{}ojasiewicz exponent of $\phi+\psi$ at $\ol x$, according to the Definition \ref{kl-phi}. 

We define $\sigma:[0,+\infty)\rightarrow[0,+\infty)$ by (see also \cite{b-d-l2006})  
$$\sigma(t)=\int_{t}^{+\infty}\|\dot x(s)\|ds +\int_{t}^{+\infty}\|\dot v(s)\|ds \ \mbox{ for all }t\geq 0.$$
It is immediate that 
\begin{equation}\label{x-sigma1}\|x(t)-\ol x\|\leq \int_{t}^{+\infty}\|\dot x(s)\|ds \ \forall t\geq 0.\end{equation}

Indeed, this follows by noticing that for $T \geq t$
\begin{align*}\|x(t)-\ol x\|= \ & \left\|x(T)-\ol x-\int_t^T\dot x(s)ds\right\|\\
                          \leq  & \ \|x(T)-\ol x\|+\int_{t}^T\|\dot x(s)\|ds,
                          \end{align*}
and by letting afterwards $T\rightarrow +\infty$.

Similarly, we have 
\begin{equation}\label{x-sigma2}\left\|v(t)-\nabla\phi(\ol x)\right\|\leq \int_{t}^{+\infty}\|\dot v(s)\|ds \ \forall t\geq 0.\end{equation}
From \eqref{x-sigma1} and \eqref{x-sigma2} we derive 
\begin{equation}\label{x-sigma}\|x(t)-\ol x\|+\left\|v(t)-\nabla\phi(\ol x)\right\|\leq \sigma(t)\ \forall t\geq 0.\end{equation}

We assume that for every $t\geq 0$ we have $(\phi+\psi)(x(t))>(\phi+\psi)(\ol x).$ As seen in the proof of 
Theorem \ref{conv-kl} otherwise the conclusion follows automatically. Furthermore, by invoking again the proof 
of Theorem \ref{conv-kl} , there exist $\varepsilon>0$, $t_0\geq 0$ and $\alpha>0$ such that for almost every $t\geq t_0$ (see \eqref{ineq-pt-conv-r2})
\begin{equation}\label{alpha-conv} \alpha \|\dot x(t)\|+\alpha \|\dot v(t)\|+ \frac{d}{dt}\Big[(\phi+\psi)(x(t))-(\phi+\psi)(\ol x)\Big]^{1-\theta}\leq 0\end{equation}
and
\begin{equation*} \left\| x(t)-\ol x\right\|<\varepsilon.\end{equation*}

We derive by integration for $T\geq t\geq t_0$
$$\alpha\int_t^{T}\|\dot x(s)\|ds+\alpha\int_t^{T}\|\dot v(s)\|ds+\Big[(\phi+\psi)(x(T))-(\phi+\psi)(\ol x)\Big]^{1-\theta} $$$$
\leq \Big[(\phi+\psi)(x(t))-(\phi+\psi)(\ol x)\Big]^{1-\theta},$$
hence 
\begin{equation}\label{ineq1}\alpha\sigma (t)\leq \Big[(\phi+\psi)(x(t))-(\phi+\psi)(\ol x)\Big]^{1-\theta} \ \forall t\geq t_0.\end{equation}

Since $\theta$ is the \L{}ojasiewicz exponent of $\phi+\psi$ at $\ol x$, we have 
$$\left|(\phi+\psi)(x(t))-(\phi+\psi)(\ol x)\right|^{\theta}\leq C\|\nabla(\phi+\psi)(x(t))\|$$
for every $t\geq t_0$. 
From the second relation in \eqref{dyn-syst2} we derive for almost every $t \in [t_0, +\infty)$
$$\left|(\phi+\psi)(x(t))-(\phi+\psi)(\ol x)\right|^{\theta}\leq C \lambda\|\dot x(t)\|+C\|\dot v(t)\|,$$ 
which combined with \eqref{ineq1} yields 
\begin{equation}\label{ineq2}\alpha\sigma (t)\leq \big(C \lambda\|\dot x(t)\|+C\|\dot v(t)\|\big)^{\frac{1-\theta}{\theta}}
\leq (C\max(\lambda,1))^{\frac{1-\theta}{\theta}}\cdot (\|\dot x(t)\|+\|\dot v(t)\|)^{\frac{1-\theta}{\theta}}.\end{equation}

Since \begin{equation}\label{dsigma}\dot \sigma (t)=-\|\dot  x(t)\|-\|\dot  v(t)\|,\end{equation} we conclude that there exists $\alpha'>0$ such that for almost every $t \in [t_0, +\infty)$
\begin{equation}\label{sigma} \dot\sigma (t)\leq -\alpha'\big(\sigma(t)\big)^{\frac{\theta}{1-\theta}}. 
\end{equation}

If $\theta=\frac{1}{2}$, then $$\dot\sigma (t)\leq -\alpha'\sigma(t)$$ for almost every $t \in [t_0, +\infty)$. By multiplying with 
$\exp(\alpha' t)$ and integrating afterwards 
from $t_0$ to $t$, it follows that there exist $a_1,b_1>0$ such that 
$$\sigma (t)\leq a_1\exp(-b_1t) \ \forall t\geq t_0$$ and the conclusion of (b) is immediate from \eqref{x-sigma}. 

Assume that $0<\theta<\frac{1}{2}$. We obtain from \eqref{sigma} 
$$\frac{d}{dt}\left(\sigma(t)^{\frac{1-2\theta}{1-\theta}}\right)\leq-\alpha' \frac{1-2\theta}{1-\theta}$$
for almost every $t \in [t_0, +\infty)$.

By integration we obtain $$\sigma(t)^{\frac{1-2\theta}{1-\theta}}\leq -\ol \alpha t+\ol \beta \ \forall t\geq t_0,$$ 
where $\ol \alpha>0$. Thus there exists $T\geq 0$ such that $$\sigma (T)\leq 0 \  \forall t\geq T,$$
which implies that $x$ and $y$ are constant on $[T,+\infty)$. 

Finally, suppose that $\frac{1}{2}<\theta<1$. We obtain from \eqref{sigma} 
$$\frac{d}{dt}\left(\sigma(t)^{\frac{1-2\theta}{1-\theta}}\right)\geq\alpha' \frac{2\theta-1}{1-\theta}$$
for almost every $t \in [t_0, +\infty)$. By integration we derive $$\sigma(t)\leq (a_2t+b_2)^{-\left(\frac{1-\theta}{2\theta-1}\right)} \ \forall t\geq t_0,$$ where 
$a_2,b_2>0$. Statement (c) follows from \eqref{x-sigma}.
\end{proof}

\end{document}